\documentclass[11pt]{amsart}

\usepackage{amscd}
\usepackage{amsmath, amssymb}
\usepackage{amsfonts}
\usepackage{enumerate}
\usepackage{esint}
\usepackage{array}
\newcommand{\vp}{\varphi}

\newcommand{\ddbar}{\sqrt{-1} \partial \overline{\partial}}

\newcommand{\ol}{\overline}

\newcommand{\dbar}{\overline{\partial}}
\newcommand{\de}{\partial}
\newcommand{\vol}{\mathrm{Vol}}

\newcommand{\ov}[1]{\overline{#1}}

\newcommand{\ti}[1]{\tilde{#1}}

\renewcommand{\leq}{\leqslant}
\renewcommand{\geq}{\geqslant}
\renewcommand{\le}{\leqslant}

\newcommand{\be}{\begin{equation}}
\newcommand{\ee}{\end{equation}}

\begin{document}
\newcounter{remark}
\newcounter{theor}
\setcounter{remark}{0}
\setcounter{theor}{1}
\newtheorem{claim}{Claim}
\newtheorem{theorem}{Theorem}[section]
\newtheorem{lemma}[theorem]{Lemma}
\newtheorem{corollary}[theorem]{Corollary}
\newtheorem{proposition}[theorem]{Proposition}
\newtheorem{question}{Question}[section]
\newtheorem{defn}{Definition}[theor]
\newtheorem{remark}[theorem]{Remark}

\numberwithin{equation}{section}

\title{The adiabatic limit of Fu-Yau equations}
\author[L. Huang]{Liding Huang}
\address{School of Mathematical Sciences, University of Science and Technology of China, Hefei 230026, P. R. China}
\email{huangld@mail.ustc.edu.cn}

\begin{abstract}
	In this paper, we consider the  adiabatic limit of Fu-Yau equations on a product of two Calabi-Yau manifolds. We prove that the  adiabatic limit of Fu-Yau equations are quasilinear equations.  
\end{abstract}
\maketitle
\section{introduction}
To study supergravity in theoretical physics,  Strominger \cite{SA} proposed a new system of equations, now referred to as the Strominger system, on $3$-dimensional complex manifolds. It can be viewed as a  generalization of the Calabi
equation for Ricci flat K\"{a}hler metric to non-K\"{a}hler spaces \cite{YT05}. 
 In \cite{FuY08}, Fu-Yau gave non-perturbative, non-K\"{a}hler solutions of the Strominger system on a toric fibration over a $K3$ surface constructed by Goldstein-Prokushki \cite{GP04} and proposed a modified Strominger system  in higher dimensions. They \cite{FuY07, FuY08} proved that it suffices to solve the Fu-Yau equation for the modified Strominger system (Strominger system) in higher dimensions (dimension 3) on the  toric fibration over a Calabi-Yau manifold and derived the existence of the Fu-Yau equation in dimension $2$. The Fu-Yau equation has been studied extensively, for example  \cite{PPZ17}, \cite{PPZ18}, \cite{CHZ19}, \cite{CHZ192}, \cite{PPZ18b} and the references therein. 
  
  In this paper, we consider the adiabatic limit  of the Fu-Yau equation. The adiabatic limit of partial differential equations (system) have been studied  extensively. For example, Donaldson-Tomas \cite{DT} considered the adiabatic limit of holomorphic instanton equations.  The adiabatic limit of anti-self-dual connections was studied by J. Chen \cite{CJ1, CJ2}  on unitary bundles over a product of two compact Riemann surfaces  and Calabi-Yau surfaces, respectively. V. Tosatti \cite{TV} researched adiabatic limits of Ricci-flat K\"{a}hler metrics on a
  Calabi-Yau manifold. J. Fine \cite{FJ04} proved the existence of constant scalar curvature K\"{a}hler metrics on certain complex surfaces, via an adiabatic limit.
   For more interesting papers about the adiabatic limit, we refer to \cite{SJ15}, \cite{SA19} and the references therein.

Now we recall the construction  of the toric fibration over a Calabi-Yau manifold in \cite{GP04}. Let $\beta_{1}$ and $\beta_{2}$ be primitive harmonic $(1,1)$ forms on an $n$-dimension Calabi-Yau manifold $(M,\omega)$  $(\omega$ is a K\"{a}hler-Ricci flat metric on $M$)  with a nowhere vanishing holomorphic $(n,0)$-form $\Omega$  satisfying: $[\frac{\beta_{1}}{2\pi}],[\frac{\beta_{2}}{2\pi}]\in H^{1,1}(M,\mathbb{Z})$. The toric fibration $\pi: X\rightarrow M$ is determined by $\beta_{1}$ and $\beta_{2}$.

In this paper, we consider $(M=M_{1}\times M_2, \omega_{\epsilon}=\epsilon^{2}\omega_{1}\oplus\omega_{2})$, where $(M_{1},\omega_{1}), (M_{2},\omega_{2})$ are $n_{1}$-dimension, $n_{2}$-dimension Ricci flat K\"{a}hler manifolds, respectively. Let $E_{i}$ be  degree zero stable holomorphic vector bundles with Hermitian-Einstein metric $H_{i}$ on $M_{i}$ with respect to $\omega_{i}, i=1,2$. Denote by $\pi_{i}:M\rightarrow M_{i}, i=1,2$  the projection map. Consider the bundle $(E=\pi_{1}^{*}E_{1}\oplus\pi_{2}^{*}E_{2}, H=\pi_{1}^{*}H_{1}\oplus \pi_{2}^{*}H_{2})$, then $H$ is a Hermitian-Einstein metric with respect to $\omega_{\epsilon}.$ Let $F$ be the Chern curvature of $H$.
 Suppose there exist $\tilde{\beta}_{1}, \tilde{\beta}_{2}\in H^{1,1}(M_{1}, \mathbb{Z}), \tilde{\beta}_{3}, \tilde{\beta}_{4}\in H^{1,1}(M_{2}, \mathbb{Z})$ that are primitive harmonic $(1,1)$ forms with respect to $\omega_{1}, \omega_{2}$, respectively. 
Set $\beta_{1}=\pi_{1}^{*}\tilde{\beta}_{1}\oplus\pi^{*}_{2}\tilde{\beta}_{3}, \beta_{2}=\pi^{*}_{1}\tilde{\beta}_{2}\oplus\pi^{*}_{2}\tilde{\beta}_{4} $, then $\beta_{1}, \beta_{2}\in H^{1,1}(M, \mathbb{Z})$ are primitive  harmonic $(1,1)$ forms with respect to $\omega_{\epsilon}$ for any $\epsilon.$ 
Then Fu-Yau equations have the following form \cite{FuY08}
\begin{equation}\label{Fu-Yau equation 0}
\begin{split}
\ddbar(e^{\vp_{\epsilon}}\omega_{\epsilon}-\alpha e^{-\vp_{\epsilon}}\rho) \wedge\omega_{\epsilon}^{n-2}
&+n\alpha(\ddbar\vp_{\epsilon})^{2}\wedge\omega_{\epsilon}^{n-2}\\
&+\mu\omega_{\epsilon}^{n}=0,
\end{split}
\end{equation}
 where $$\rho=\frac{1}{2}\sqrt{-1}tr(\partial B\wedge\ol{\partial}B^{*}\cdot g^{-1}_{\epsilon})$$ are smooth $(1,1)$ forms and
  $$\mu=\frac{-\alpha(tr R_{\epsilon}\wedge R_{\epsilon}-tr F\wedge F)\wedge\omega_{\epsilon}^{n-2}}{\omega_{\epsilon}^{n}}+\frac{\beta_{1}^{2}\wedge\omega_{\epsilon}^{n-2}}{\omega_{\epsilon}^{n}}+\frac{\beta_{2}^{2}\wedge\omega^{n-2}_{\epsilon}}{\omega_{\epsilon}^{n}}$$
are smooth functions satisfying $\int_{M}\mu\omega^{n}=0$.
It follows
 \begin{equation*}
\alpha=\left(\int_{M}\Big(tr R_{\epsilon}\wedge R_{\epsilon}-tr F\wedge F\Big)\wedge\omega_{\epsilon}^{n-2}\right)^{-1}\int_{M}\Big(\beta_{1}^{2}\wedge\omega_{\epsilon}^{n-2}+\beta_{2}^{2}\wedge\omega^{n-2}_{\epsilon}\Big).
\end{equation*}

 Denote by $R_{\epsilon}$ the curvature tensor of $\omega_{\epsilon}$. It is known that $B$ only depends on $\beta_{1}, \beta_{2}$ (see \cite[p.380]{FuY08}) and $\partial B\wedge\ol{\partial}B^{*}\cdot g^{-1}_{\epsilon}$ is a $(1,1)$ form with value in the endomorphism bundle of $TM$. $B^{*}$ is the  conjugate transpose of $B$.   By tr, we denote the trace of
the endomorphism bundle of $TM$ or $E$.  Since $tr R_{\epsilon}\wedge R_{\epsilon}$ and $F$ are independent of $\epsilon$,
\begin{equation}\label{order}
 \|\rho\|_{C^{0}(\omega)}\leq C\epsilon^{-2}, \|\mu\|_{L^{\infty}}\leq C\epsilon^{-t},
 \end{equation}
where $\|\rho\|_{C^{0}(\omega)}$ means the $C^{0}$ norm of $\rho$ with respect to $\omega$,  $\omega=\omega_{1}\oplus\omega_{2}$ and $t\geq4$ is a constant.

We consider general Fu-Yau equations,
 \begin{equation}\label{Fu-Yau equation 1}
\begin{split}
\ddbar(e^{\vp_{\epsilon}}\omega_{\epsilon}-\alpha e^{-\vp_{\epsilon}}\rho_{\epsilon}) \wedge\omega_{\epsilon}^{n-2}
&+n\alpha(\ddbar\vp_{\epsilon})^{2}\wedge\omega_{\epsilon}^{n-2}\\
&+\mu_{\epsilon}\omega_{\epsilon}^{n}=0,
\end{split}
\end{equation}
where $\rho_{\epsilon}$ is a $(1,1)$ form and $\mu_{\epsilon}$ is a smooth function satisfying $\int_{M}\mu_{\epsilon}\omega_{\epsilon}^{n}=0$. Suppose there exist $t\geq 4$ and uniform constants $K$ such that 
\begin{equation}\label{uniform function}
\|\epsilon^{t}\mu_{\epsilon}\|_{C^{2}(\omega_{\epsilon})}\leq K, \|\epsilon^{t}\rho_{\epsilon}\|_{C^{2}(\omega_{\epsilon})}\leq K, |\epsilon^{t-4}\alpha|\leq K.
\end{equation}

Using the argument of \cite{CHZ19}, we have the following existences and a priori estimates of solutions:
\begin{theorem}\label{Estimate}
	There exist  solutions  $\vp_{\epsilon}$ of \eqref{Fu-Yau equation 1} satisfying the elliptic condition
	\begin{equation}\label{Elliptic condition}
	\tilde{\omega}_{\epsilon} = e^{\vp_{\epsilon}}\omega_{\epsilon}+\alpha e^{-\vp_{\epsilon}}\rho_{\epsilon}+2n\alpha\ddbar\vp_{\epsilon} \in \Gamma_{2}(M),
	\end{equation}
	and the normalization condition
	\begin{equation}\label{Normalization condition}
	\|e^{-\vp_{\epsilon}}\|_{L^{1}(\omega_{\epsilon})} = A\epsilon^{2n_{1}+t},
	\end{equation}
	where $\Gamma_{2}(M)$ is the space of the second convex $(1,1)$-forms. In addition, we can find a uniform small constant $A_{0}>0$  and a uniform constant $M_{0}$ depending only on  $K$ and $(M,\omega)$ such that for any positive $A\leq  A_{0} $, we have
	\begin{equation*}
	\frac{1}{M_{0}A\epsilon^{t}}\leq e^{\vp_{\epsilon}}\leq \frac{M_{0}}{A\epsilon^{t}}
	\end{equation*}
	and  	\begin{equation*}
	|\de\vp|_{\omega_{\epsilon}}+ |\ddbar\vp|_{\omega_{\epsilon}}\leq M_{0}.
	\end{equation*}
\end{theorem}

Note there exists a subsequence of $\epsilon^{t}\mu_{\epsilon}$, still denoted by $\epsilon^{t}\mu_{\epsilon}$, weakly converge to a function $\mu_{0}$ in $L^{2}(M,\omega)$. 
 We further prove
\begin{theorem}\label{convergence}
	Let $\vp_{\epsilon}$ be the solution of \eqref{Fu-Yau equation 1} defined in Theorem \ref{Estimate}.
	There exists a subsequence, still denoted by $\vp_{\epsilon}$, such that $\vp_{\epsilon}+t\ln\epsilon$  converge to a function $ \hat{\vp}$ in $C^{1,\beta}$ on $M_{2}$ which is a weak solution of
\begin{equation}\label{limit of equation}
\ddbar (e^{\hat{\vp}})\wedge\omega_{2}^{n_{2}-1}+\hat{\mu}\omega_{2}^{n_{2}}=0,
\end{equation}
where $\hat{\mu}=\frac{\int_{M_{1}}\mu_{0}\omega_{1}^{n_{1}}}{\int_{M_{1}}\omega_{1}^{n_{1}}}$.
\end{theorem}

For \eqref{Fu-Yau equation 0}, choose $t$ such that 
\begin{equation*}
\|\epsilon^{t}\mu\|_{C^{2}(\omega_{\epsilon})}\leq K, \|\epsilon^{t}\rho\|_{C^{2}(\omega_{\epsilon})}\leq K, |\epsilon^{t-4}\alpha|\leq K .
\end{equation*}
 we have
\begin{equation*}
\begin{split}
\epsilon^{t}\mu\rightarrow\mu_{0}=&\frac{-\epsilon^{t-4}\alpha_{0}(tr R_{1\epsilon}\wedge R_{1\epsilon}-tr F_{1}\wedge F_{1})\wedge\omega_{1}^{n-2}}{\omega_{1}^{n}}\\
&+\epsilon^{t-4}\big(\frac{\tilde{\beta}_{1}^{2}\wedge\omega_{1}^{n_{1}-2}}{\omega_{1}^{n}}
+\frac{\tilde{\beta}_{2}^{2}\wedge\omega_{1}^{n_{1}-2}}{\omega_{1}^{n_{1}}}\big),
\end{split}
\end{equation*}
where $\alpha_{0}$ is a constant such that $\int_{M}\mu_{0}\omega^{n}=0.$ It is  a function on $M_{1}$, therefore $\hat{\mu}$ is a constant. Note $\int_{M}\mu_{0}\omega^{n}=0$. We have $\hat{\mu}=0$.  It follows from Theorem \eqref{convergence}:

	\begin{corollary}
		Let $\vp_{\epsilon}$ be the solution of \eqref{Fu-Yau equation 0} defined in Theorem \ref{Estimate}.
		There exists a subsequence, still denoted by $\vp_{\epsilon}$, such that $\vp_{\epsilon}+t\ln\epsilon$  converge to a function $ \hat{\vp}$ in $C^{1,\beta}$ on $M_{2}$ which is a weak solution of
		\begin{equation*}
		\ddbar (e^{\hat{\vp}})\wedge\omega_{2}^{n_{2}-1}=0.
		\end{equation*}
		It follows $\hat{\vp}$ is a constant.
	\end{corollary}

%

In section 2, we give the uniform estimates,  using the Chu-Huang-Zhu's argument \cite{CHZ19}. Compared to \cite{CHZ19}, there are something that are different.
First, the Sobolev inequality with respect to $\omega_{\epsilon}$  depends on $\epsilon$. 
In addition, there are terms in equations \eqref{Fu-Yau equation 1} involving $\frac{1}{\epsilon}$, we need some delicate calculations in the $C^{0}$ estimates. 
	In fact, we consider the function $\vp_{\epsilon}+\ln{\epsilon}^{t}$ which satisfy the Fu-Yau equation that the terms in equation are uniformly bounded for $\epsilon$, but $\alpha$ will tend to $0$.  It is also important to choose  an appropriate  $L^{1}$ norm of $e^{-\vp_{\epsilon}}$.  
	  When we obtain the $C^{0}$ estimates, we can use the Chu-Huang-Zhu's result to prove the $C^{1}$ and $C^{2}$ estimates. However, there are $\frac{1}{\alpha}$ in the proof of \cite{CHZ19}.	  We show that  Chu-Huang-Zhu's $C^{1}$ and $C^{2}$ estimates are uniform when $\alpha\rightarrow 0$ (see subsection 2.2). 
	  
	  In section 3, to prove the weak convergence, we use an argument of Tosatti \cite{TV}. By the $L^{\infty}$ estimates for Monge-Amp\`ere equation \cite{KN15}, we can conclude that $\vp_{\epsilon}-\frac{1}{Vol(M_{1})}\int_{M_{1}}\vp_{\epsilon}\omega_{1}^{n_1}$ tend to zero. Then, we can prove that a subsequence of $\vp_{\epsilon}$ converge to a function on $M_{2}$, by Theorem \ref{Estimate}.
	  
	\vspace{.2cm}
	{\it Acknowledgements. }
	 I would like to thank Professor J. Chen
	for suggesting this problem and for very useful	discussions.
	This work was carried out while   I was visiting the Department of Mathematics at the University of British Columbia, supported by the China Scholarship Council (File No. 201906340217). I would like to thank UBC for the hospitality and support.

\section{Proof of Theorem \ref{Estimate}}
In this section, we prove Theorem \ref{Estimate}. 
\subsection{$C^{0}$ estimates}First, we give the zero order estimates.
\begin{proposition}\label{C0 estimate}
	Let $\vp_{\epsilon}$ be a smooth solution of (\ref{Fu-Yau equation 1}). There exist constants $A_{0}$ and $M_{0}$ depending only on $K$ and $(M,\omega)$ such that if
 	\begin{equation}
	e^{-\vp_{\epsilon}} \leq \epsilon^{t}\delta_{0} := \epsilon^{t}\sqrt{\frac{1}{2\epsilon^{2t}|\alpha|\|\rho_{\epsilon}\|_{C^{0}(\omega_{\epsilon})}+1}}
	\end{equation}
	and
	\begin{equation}\label{normalized}
	 \|e^{-\vp_{\epsilon}}\|_{L^{1}(\omega_{\epsilon})} = A\epsilon^{2n_{1}+t} \leq A_{0}\epsilon^{2n_{1}+t},
	\end{equation}
we have
\begin{equation*}
\frac{1}{M_{0}A\epsilon^{t}}\leq e^{\vp_{\epsilon}}\leq \frac{M_{0}}{A\epsilon^{t}}.
\end{equation*}
\end{proposition}

We need the following the Sobolev inequality and the Poincar\'e inequality:
\begin{lemma}
	For any $u\in C^{\infty}(M)$, there exists a uniform constant $C$ such that
\begin{equation}\label{Sobolev inequality}
	\|u\|_{L^{\frac{2np}{2n-p}}(\omega_{\epsilon})}\leq C\epsilon^{-\frac{n_{1}}{n}}(\|u\|_{L^{p}(\omega_{\epsilon})}+\|\partial u\|_{L^{p}(\omega_{\epsilon})}),
\end{equation}
and
\begin{equation}\label{poincare}
\int_{M}\big|u-\frac{1}{\text{Vol}_{\epsilon}(M)}\int_{M}u\omega_{\epsilon}^{n}\big|^{2}\omega_{\epsilon}^{n}\leq C\int_{M}|\partial u|_{\epsilon}^{2}\omega^n_{\epsilon}
\end{equation}
where 	$\|u\|_{L^{\frac{2np}{2n-p}}(\omega_{\epsilon})}=(\int_{M}|u|^{\frac{2np}{2n-p}}\omega_{\epsilon}^{n})^{\frac{2n-p}{2np}}$ and $\text{Vol}_{\epsilon}(M)=\int_{M}\omega_{\epsilon}^{n}$. Here $|\partial u|_{\epsilon}$ is the norm of  gradient of $u$ with respect to $\omega_{\epsilon}$ and $\|\partial u\|_{L^{p}(\omega_{\epsilon})}=(\int_{M}|\partial u|_{\epsilon}^{p}\omega^n_{\epsilon})^{\frac{1}{p}}$ .
\end{lemma}
\begin{proof}
By the Sobolev inequality with respect to $\omega$, 	
\begin{equation*}
\begin{split}
\|u\|_{L^{\frac{2np}{2n-p}}(\omega_{\epsilon})}\leq& \epsilon^{\frac{n_{1}(2n-p)}{np}}\|u\|_{L^{\frac{2np}{2n-p}}(\omega)}\\
                                               \leq &C\epsilon^{\frac{n_{1}(2n-p)}{np}}(\|u\|_{L^{p}(\omega)}+\|\partial u\|_{L^{p}(\omega)})\\
                                               \leq &C\epsilon^{\frac{n_{1}(2n-p)}{np}-\frac{2n_{1}}{p}}(\|u\|_{L^{p}(\omega_{\epsilon})}+\|\partial u\|_{L^{p}(\omega_{\epsilon})})  \\
                                               =&C\epsilon^{-\frac{n_{1}}{n}}(\|u\|_{L^{p}(\omega_{\epsilon})}+\|\partial u\|_{L^{p}(\omega_{\epsilon})}).
\end{split}
\end{equation*}
Denote $V=\int_{M}\omega^{n}$.
We have, by the Poincar\'e inequality with respect to $\omega$, 
\begin{equation*}
\begin{split}
\int_{M}|u-\frac{1}{\text{Vol}_{\epsilon}(M)}\int_{M}u\omega_{\epsilon}^{n}|^{2}\omega_{\epsilon}^{n}&=\epsilon^{2n_{1}}\left(\int_{M}u^{2}\omega^{n}-\frac{1}{V}(\int_{M}u\omega^{n})^{2}\right)\\
&=\epsilon^{2n_{1}}\int_{M}|u-\frac{1}{V}\int_{M}u\omega^{n}|^{2}\omega^{n}\\
&\leq C\epsilon^{2n_{1}}\int_{M}|\partial u|^{2}_{\omega}\omega^{n}\\
&\leq C\int_{M}|\partial u|^{2}_{\epsilon}\omega_{\epsilon}^{n},
\end{split}
\end{equation*}
where $|\partial u|^{2}_{\omega}$ is the norm of gradient of $u$ with respect to $\omega.$
\end{proof}

Now, we complete the proof of the Proposition \ref{C0 estimate}. We will use the argument in \cite[Proposition 2.1]{CHZ19}.
\begin{proof}[Proof of Proposition \ref{C0 estimate}]
	Multiplying the equation \eqref{Fu-Yau equation 1} by $\epsilon^{t}$, we have
	
	\begin{equation}\label{Fu-Yau equation 2}
	\begin{split}
	\ddbar(e^{\tilde{\vp}_{\epsilon}}\omega_{\epsilon}-\tilde{\alpha} e^{-\tilde{\vp}_{\epsilon}}\tilde{\rho}_{\epsilon}) \wedge\omega_{\epsilon}^{n-2}
	&+n\tilde{\alpha}(\ddbar\tilde{\vp}_{\epsilon})^{2}\wedge\omega_{\epsilon}^{n-2}\\
	&+\tilde{\mu}_{\epsilon}\omega_{\epsilon}^{n}=0,
	\end{split}
	\end{equation}
where $\tilde{\alpha}=\epsilon^{t}\alpha$, $\tilde{\vp}_{\epsilon}=\vp_{\epsilon}+t\ln\epsilon$, $\tilde{\rho}_{\epsilon}=\epsilon^{t}\rho_{\epsilon}$ and $\tilde{\mu}_{\epsilon}=\epsilon^{t}\mu_{\epsilon}$. By \eqref{uniform function},
\begin{equation}\label{modified function} 
\|\tilde{\rho}_{\epsilon}\|_{C^{2}(\omega_{\epsilon})}+\|\tilde{\mu}_{\epsilon}\|_{C^{2}(\omega_{\epsilon})}\leq 2K,\ \  \tilde{\alpha}\rightarrow 0.
\end{equation}
Then we have, by \eqref{normalized},
	\begin{equation}\label{Normalization condition 1}
\|e^{-\tilde{\vp}_{\epsilon}}\|_{L^{1}(\omega_{\epsilon})} = A\epsilon^{2n_{1}}.
\end{equation}
 First, we estimate  the infimum. 
By Stokes' formula,  we derive
\begin{equation}\label{Infimum estimate equation 5}
\begin{split}
&-2k\int_{M}e^{-k\tilde{\vp}_{\epsilon}}(e^{\tilde{\vp}_{\epsilon}}\omega_{\epsilon}+\tilde{\alpha} e^{-\tilde{\vp}_{\epsilon}}\tilde{\rho}_{\epsilon})\wedge\sqrt{-1}\de\tilde{\vp}_{\epsilon}\wedge\dbar\tilde{\vp}_{\epsilon}\wedge\omega_{\epsilon}^{n-2} \\
=&-2k\int_{M}e^{-k\tilde{\vp}_{\epsilon}}\sqrt{-1}\de\tilde{\vp}_{\epsilon}\wedge\dbar(e^{\tilde{\vp}_{\epsilon}}\omega_{\epsilon}-\tilde{\alpha} e^{-\tilde{\vp}_{\epsilon}}\tilde{\rho}_{\epsilon})\wedge\omega_{\epsilon}^{n-2}\\
&-2\ti{\alpha} k\int_{M}e^{-(k+1)\tilde{\vp}_{\epsilon}}\sqrt{-1}\de\tilde{\vp}_{\epsilon} \wedge\dbar\tilde{\rho}_{\epsilon}\wedge\omega_{\epsilon}^{n-2}\\
=&-2\int_{M}e^{-k\tilde{\vp}_{\epsilon}}\ddbar(e^{\tilde{\vp}_{\epsilon}}\omega_{\epsilon}-\ti{\alpha} e^{-\tilde{\vp}_{\epsilon}}\tilde{\rho}_{\epsilon})\wedge\omega_{\epsilon}^{n-2}\\
&-2\ti{\alpha} k\int_{M}e^{-(k+1)\tilde{\vp}_{\epsilon}}\sqrt{-1}\de\tilde{\vp}_{\epsilon} \wedge\dbar\tilde{\rho}_{\epsilon}\wedge\omega_{\epsilon}^{n-2}.\\
\end{split}
\end{equation}
Now we deal with the first term on the   right hand side  in \eqref{Infimum estimate equation 5}.

By (\ref{Elliptic condition}), it follows for $k\geq 2$,
\begin{equation*}
k\int_{M}e^{-k\tilde{\vp}_{\epsilon}}\sqrt{-1}\de\tilde{\vp}_{\epsilon}\wedge\dbar\tilde{\vp}_{\epsilon}\wedge\ti{\omega}_{\epsilon}\wedge\omega_{\epsilon}^{n-2} \geq 0.
\end{equation*}
Using \eqref{Fu-Yau equation 2}, \eqref{Elliptic condition} and integration by parts,   we have
\begin{equation}\label{Infimum estimate equation 2}
\begin{split}
&-2\int_{M}e^{-k\tilde{\vp}_{\epsilon}}\ddbar(e^{\tilde{\vp}_{\epsilon}}\omega_{\epsilon}-\tilde{\alpha} e^{-\tilde{\vp}_{\epsilon}}\tilde{\rho}_{\epsilon})\wedge\omega_{\epsilon}^{n-2}\\
=&2n\tilde{\alpha} \int_{M}e^{-k\tilde{\vp}_{\epsilon}}\ddbar\tilde{\vp}_{\epsilon}\wedge\ddbar\tilde{\vp}_{\epsilon}\wedge\omega_{\epsilon}^{n-2}+2\int_{M}e^{-k\tilde{\vp}_{\epsilon}}\ti{\mu}_{\epsilon}\frac{\omega_{\epsilon}^{n}}{n!}\\
=&-2n\tilde{\alpha} \int_{M}\sqrt{-1}\de e^{-k\tilde{\vp}_{\epsilon}}\wedge\dbar\tilde{\vp}_{\epsilon}\wedge\ddbar\tilde{\vp}_{\epsilon}\wedge\omega_{\epsilon}^{n-2}+2\int_{M}e^{-k\tilde{\vp}_{\epsilon}}\ti{\mu}_{\epsilon}\frac{\omega_{\epsilon}^{n}}{n!}\\
=&2n\tilde{\alpha} k\int_{M}e^{-k\tilde{\vp}_{\epsilon}}\sqrt{-1}\de\tilde{\vp}_{\epsilon}\wedge\dbar\tilde{\vp}_{\epsilon}\wedge\ddbar\tilde{\vp}_{\epsilon}\wedge\omega_{\epsilon}^{n-2}+2\int_{M}e^{-k\tilde{\vp}_{\epsilon}}\ti{\mu}_{\epsilon}\frac{\omega_{\epsilon}^{n}}{n!}\\
\geq&-k\int_{M}e^{-k\tilde{\vp}_{\epsilon}}(e^{\tilde{\vp}_{\epsilon}}\omega_{\epsilon}+\tilde{\alpha}
 e^{-\tilde{\vp}_{\epsilon}}\tilde{\rho}_{\epsilon})\wedge\sqrt{-1}\de\tilde{\vp}_{\epsilon}\wedge\dbar\tilde{\vp}_{\epsilon}\wedge\omega_{\epsilon}^{n-2}+2\int_{M}e^{-k\tilde{\vp}_{\epsilon}}\ti{\mu}_{\epsilon}\frac{\omega_{\epsilon}^{n}}{n!} .\\
\end{split}
\end{equation}

Substituting (\ref{Infimum estimate equation 2}) into (\ref{Infimum estimate equation 5}),  we obtain
\begin{equation}\label{Infimum estimate equation 3}
\begin{split}
& k\int_{M}e^{-k\tilde{\vp}_{\epsilon}}(e^{\tilde{\vp}_{\epsilon}}\omega_{\epsilon}+\ti{\alpha} e^{-\tilde{\vp}_{\epsilon}}\tilde{\rho}_{\epsilon})\wedge\sqrt{-1}\de\tilde{\vp}_{\epsilon}\wedge\dbar\tilde{\vp}_{\epsilon}\wedge\omega_{\epsilon}^{n-2} \\
\leq {} & 2\ti{\alpha} k\int_{M}e^{-(k+1)\tilde{\vp}_{\epsilon}}\sqrt{-1}\de\tilde{\vp}_{\epsilon} \wedge\dbar\tilde{\rho}_{\epsilon}\wedge\omega_{\epsilon}^{n-2}-2\int_{M}e^{-k\tilde{\vp}_{\epsilon}}\tilde{\mu}_{\epsilon}\frac{\omega_{\epsilon}^{n}}{n!}.
\end{split}
\end{equation}
From the definition of  $\delta_{0}$ and  the condition $e^{-\tilde{\vp}_{\epsilon}}\leq\delta_{0}$, we get
\begin{equation}\label{Infimum estimate equation 4}
\omega_{\epsilon}+\ti{\alpha} e^{-2\ti{\vp}_{\epsilon}}\ti{\rho}_{\epsilon}\geq\frac{1}{2}\omega_{\epsilon}.
\end{equation}
By \eqref{modified function}, \eqref{Infimum estimate equation 3}, \eqref{Infimum estimate equation 4} and the Cauchy-Schwarz inequality, we conclude
\begin{equation*}
\begin{split}
k\int_{M}e^{-(k-1)\tilde{\vp}_{\epsilon}}|\de\tilde{\vp}_{\epsilon}|_{\epsilon}^{2}\omega_{\epsilon}^{n}
\leq {} & Ck\int_{M}\left(e^{-(k+1)\tilde{\vp}_{\epsilon}}|\de\tilde{\vp}_{\epsilon}|_{\epsilon}+e^{-k\tilde{\vp}_{\epsilon}}\right)\omega_{\epsilon}^{n} \\
\leq {} & \frac{k}{2}\int_{M}e^{-(k-1)\tilde{\vp}_{\epsilon}}|\de\tilde{\vp}_{\epsilon}|_{\epsilon}^{2}\omega_{\epsilon}^{n}\\
&+Ck\int_{M}\left(e^{-(k+3)\tilde{\vp}_{\epsilon}}+e^{-k\tilde{\vp}_{\epsilon}}\right)\omega_{\epsilon}^{n}.
\end{split}
\end{equation*}
Note $e^{-\tilde{\vp}_{\epsilon}}\leq\delta_{0}$. It follows
\begin{equation*}
\frac{k}{2}\int_{M}e^{-(k-1)\tilde{\vp}_{\epsilon}}|\de\tilde{\vp}_{\epsilon}|_{\epsilon}^{2}\omega_{\epsilon}^{n}
\leq Ck(\delta_{0}^{4}+\delta_{0})\int_{M}e^{-(k-1)\tilde{\vp}_{\epsilon}}\omega_{\epsilon}^{n},
\end{equation*}
which implies
\begin{equation*}
\int_{M}|\de e^{-\frac{(k-1)\tilde{\vp}_{\epsilon}}{2}}|_{\epsilon}^{2}\omega_{\epsilon}^{n} \leq Ck^{2}\int_{M}e^{-(k-1)\tilde{\vp}_{\epsilon}}\omega_{\epsilon}^{n}.
\end{equation*}
Replacing $k-1$ by $k$, by \eqref{Sobolev inequality}, for $k\geq 1$, 

\[\|e^{-\tilde{\vp}_{\epsilon}}\|_{L^{\frac{kn}{n-1}}(\omega_{\epsilon})}\leq C^{\frac{1}{k}}\epsilon^{-2n_{1}\frac{1}{kn}}k^{2\frac{1}{k}}\|e^{-\tilde{\vp}_{\epsilon}}\|_{L^{k}(\omega_{\epsilon})}.\]
Using the Morse iteration and \eqref{Normalization condition 1}, we obtain 
\begin{equation}\label{infimum estimate}
\|e^{-\tilde{\vp}_{\epsilon}}\|_{L^{\infty}}\leq C \epsilon^{-2n_{1}}\|e^{-\tilde{\vp}_{\epsilon}}\|_{L^{1}(\omega_{\epsilon})}\leq CA.
\end{equation}

Now we estimate  the supremum.
By the similar calculation of (\ref{Infimum estimate equation 5})-(\ref{Infimum estimate equation 3}), for $k\geq1$, we have
\begin{equation*}
\begin{split}
& k\int_{M}e^{k\tilde{\vp}_{\epsilon}}(e^{\tilde{\vp}_{\epsilon}}\omega_{\epsilon}+\tilde{\alpha} e^{-\tilde{\vp}_{\epsilon}}\tilde{\rho}_{\epsilon})\wedge\sqrt{-1}\de\tilde{\vp}_{\epsilon}\wedge\dbar\tilde{\vp}_{\epsilon}\wedge\omega_{\epsilon}^{n-2} \\
\leq {} & 2\tilde{\alpha} k\int_{M}e^{(k-1)\tilde{\vp}_{\epsilon}}\sqrt{-1}\de\tilde{\vp}_{\epsilon} \wedge\dbar\tilde{\rho}_{\epsilon}\wedge\omega_{\epsilon}^{n-2}+2\int_{M}e^{k\tilde{\vp}_{\epsilon}}\tilde{\mu}_{\epsilon}\frac{\omega_{\epsilon}^{n}}{n!}.
\end{split}
\end{equation*}
Combining this with (\ref{Infimum estimate equation 4}), we obtain
\begin{equation*}
\int_{M}e^{(k+1)\tilde{\vp}_{\epsilon}}|\de\tilde{\vp}_{\epsilon}|_{g_{\epsilon}}^{2}\omega_{\epsilon}^{n} \leq C\int_{M}\left(e^{(k-1)\tilde{\vp}_{\epsilon}}|\de\tilde{\vp}_{\epsilon}|_{g_{\epsilon}}+e^{k\tilde{\vp}_{\epsilon}}\right)\omega_{\epsilon}^{n}.
\end{equation*}
Then, by $e^{-\tilde{\vp}_{\epsilon}}\leq\delta_{0}$ and the Cauchy-Schwarz inequality, 
\begin{equation}\label{Supremum estimate equation 2}
\int_{M}e^{(k+1)\tilde{\vp}_{\epsilon}}|\de\tilde{\vp}_{\epsilon}|_{g_{\epsilon}}^{2}\omega_{\epsilon}^{n} \leq C\int_{M}e^{k\tilde{\vp}_{\epsilon}}\omega_{\epsilon}^{n}.
\end{equation}
We use \eqref{infimum estimate} to get
\begin{equation*}
\int_{M}e^{k\tilde{\vp}_{\epsilon}}|\de\tilde{\vp}_{\epsilon}|_{g_{\epsilon}}^{2}\omega_{\epsilon}^{n} \leq C\int_{M}e^{k\tilde{\vp}_{\epsilon}}\omega_{\epsilon}^{n}.
\end{equation*}
By \eqref{Sobolev inequality}, it follows $\|e^{\tilde{\vp}_{\epsilon}}\|_{L^{\frac{kn}{n-1}}(\omega_{\epsilon})}\leq C\epsilon^{-\frac{2n_{1}}{kn}}\|e^{\tilde{\vp}_{\epsilon}}\|_{L^{k}(\omega_{\epsilon})}.$
By the Morse iteration, 
\begin{equation}\label{Supremum estimate equation 1}
\|e^{\tilde{\vp}_{\epsilon}}\|_{L^{\infty}}\leq C\epsilon^{-2n_{1}}\|e^{\tilde{\vp}_{\epsilon}}\|_{L^{1}(\omega_{\epsilon})}.
\end{equation}
Now it suffices to prove the $L^{1}$ estimate.

Without loss of generality, we assume that $\vol(M,\omega_{\epsilon})=\epsilon^{2n_{1}}$.
We define a  set by
\begin{equation*}
U = \{x\in M~|~e^{-\tilde{\vp}_{\epsilon}(x)}\geq\frac{A}{2}\}.
\end{equation*}
Then by (\ref{infimum estimate}) and  (\ref{Normalization condition 1}),   we have
\begin{equation*}
\begin{split}
\epsilon^{2n_{1}}A =    {} & \int_{U}e^{-\tilde{\vp}_{\epsilon}}\omega_{\epsilon}^{n}+\int_{M\setminus U}e^{-\tilde{\vp}_{\epsilon}}\omega_{\epsilon}^{n} \\
\leq {} & e^{-\inf_{M}\tilde{\vp}_{\epsilon}}\vol_{\epsilon}(U)+\frac{A}{2}(\epsilon^{2n_{1}}-\vol_{\epsilon}(U)) \\
\leq {} & \big(C-\frac{1}{2}\big)A\vol_{\epsilon}(U)+\epsilon^{2n_{1}}\frac{A}{2},
\end{split}
\end{equation*}
where $\vol_{\epsilon}$ is the volume of $U$ with respect to $\omega_{\epsilon}.$ 
It implies
\begin{equation}\label{Supremum estimate equation 3}
\vol_{\epsilon}(U) \geq \frac{\epsilon^{2n_{1}}}{C_{0}}.
\end{equation}

By  the Poincar\'{e} inequality and (\ref{Supremum estimate equation 2}) (taking $k=1$), we have
\begin{equation}
\int_{M}e^{2\vp}\omega_{\epsilon}^{n}-\frac{1}{\epsilon^{2n_{1}}}\big(\int_{M}e^{\vp}\omega_{\epsilon}^{n}\big)^{2}
\leq C\int_{M}|\partial e^{\vp}|_{\epsilon}^{2}\omega_{\epsilon}^{n}
\leq C\int_{M}e^{\tilde{\vp}_{\epsilon}}\omega_{\epsilon}^{n}.\notag
\end{equation}
By (\ref{Supremum estimate equation 3}) and the Cauchy-Schwarz inequality, we obtain
\begin{equation*}
\begin{split}
\big(\int_{M}e^{\tilde{\vp}_{\epsilon}}\omega_{\epsilon}^{n}\big)^{2}
& \leq (1+C_{0})\big(\int_{U}e^{\tilde{\vp}_{\epsilon}}\omega_{\epsilon}^{n}\big)^{2}
+\big(1+\frac{1}{C_{0}}\big)\big(\int_{M\setminus U}e^{\tilde{\vp}_{\epsilon}}\omega_{\epsilon}^{n}\big)^{2}\\
& \leq \frac{4(1+C_{0})}{A^{2}}(\vol_{\epsilon}(U))^{2}+\big(1+\frac{1}{C_{0}}\big)(\epsilon^{2n_{1}}-\vol_{\epsilon}(U))\int_{M}e^{2\tilde{\vp}_{\epsilon}}\omega_{\epsilon}^{n}\\
& \leq \frac{4(1+C_{0})}{A^{2}}\epsilon^{4n_{1}}+\big(1-\frac{1}{C_{0}^{2}}\big)\Big(\big(\int_{M}e^{\tilde{\vp}_{\epsilon}}\omega_{\epsilon}^{n}\big)^{2}
+C\epsilon^{2n_{1}}\int_{M}e^{\tilde{\vp}_{\epsilon}}\omega_{\epsilon}^{n}\Big).
\end{split}
\end{equation*}
It follows
\begin{equation}\label{l1-estimate}
 \|e^{\tilde{\vp}_{\epsilon}}\|_{L^{1}}\leq\frac{C\epsilon^{2n_1}}{A}.
\end{equation}


By \eqref{l1-estimate} and \eqref{Supremum estimate equation 1},  we see that $\|e^{\tilde{\vp}_{\epsilon}}\|_{L^{\infty}}\leq C\epsilon^{-2n_{1}}\|e^{\tilde{\vp}_{\epsilon}}\|_{L^{1}}\leq\frac{C}{A}.$
Note $e^{\vp_{\epsilon}}=\epsilon^{-t}e^{\tilde{\vp}_{\epsilon}}$. We complete the proof.
\end{proof}

\subsection{$C^{1}$ and $C^{2}$ estimates}
In this subsection, we recall the $C^{1}$ and $C^{2}$ estimates in \cite{CHZ19}.  In fact, their estimates are uniform for $\alpha$ when $|\alpha|$ is bounded.  However there are some $\frac{1}{\alpha}$ in the proof.  In fact, the $\frac{1}{\alpha}$ can be cancelled. In the following, we verify this in general manifolds. 
Let $(M, \omega)$ be an $n$-dimension K\"{a}hler manifold. In this subsection, we consider the following equation
\begin{equation}\label{Fu-Yau equation A}
\begin{split}
\ddbar(e^{\vp}\omega & -\alpha e^{-\vp}\rho) \wedge\omega^{n-2} \\
& +n\alpha\ddbar\vp\wedge\ddbar\vp\wedge\omega^{n-2}+\mu\frac{\omega^{n}}{n!}=0,
\end{split}
\end{equation}
where $\rho$ is a real-valued smooth $(1,1)$-form and $\mu$ is a smooth function.

\begin{proposition}\label{Gradient estimate}
	Let $\vp$ be a solution of \eqref{Fu-Yau equation A} satisfying 
		\begin{equation}\label{Elliptic condition A}
	\tilde{\omega} = e^{\vp_{\epsilon}}\omega+\alpha e^{-\vp}\rho+2n\alpha\ddbar\vp \in \Gamma_{2}(M).
	\end{equation}
	Assume that
	\begin{equation*}
	\frac{A}{M_0}\leq e^{-\vp}\leq M_0A  \text{~and~} |\de\dbar\vp|_{g}\leq D,
	\end{equation*}
	where $M_0$ is a uniform constant.  Then there exists a uniform constant $C_{0}$ such that if
	\begin{equation*}
	A \leq A_{D} := \frac{1}{C_{0}M_{0}D},
	\end{equation*}
	then
	\begin{equation*}
	|\de\vp|_{g}^{2} \leq C_{1},
	\end{equation*}
	where $C_{1}$ is a uniform constant for $\alpha$.
\end{proposition}
\begin{proof}
	The  proof is similar to \cite[Proposition 3.1]{CHZ19}. It suffices to show that the \textquotedblleft$C$\textquotedblright, from \cite[(3.9)]{CHZ19} to \cite[(3.11)]{CHZ19}, do not depend on $\frac{1}{\alpha}$. For reader's convenience, we include the proof.
	We apply the maximum principle to the quantity
	\begin{equation*}
	Q = \log|\de\vp|_{g}^{2}+\frac{\vp}{B},
	\end{equation*}
	where $B>1$ is a large uniform constant to be determined later.
	
	Assume that $Q$ achieves its maximum at $x_{0}$.
	Let $\{e_{i}\}_{i=1}^{n}$ be a local unitary frame in a neighborhood of $x_{0}$ such that, at $x_{0}$, with respect to $g$
	\begin{equation}\label{tilde gij}
	\tilde{g}_{i\ol{j}}
	= \delta_{i\ol{j}}\tilde{g}_{i\ol{i}}
	= \delta_{i\ol{j}}(e^{\vp}+\alpha e^{-\vp}\rho_{i\ov{i}}+2n\alpha \vp_{i\ov{i}}).
	\end{equation}
	For convenience, we use the following notation:
	\begin{equation*}
	\hat{\omega} = e^{-\vp}\ti{\omega},
	\hat{g}_{i\ov{j}} = e^{-\vp}\ti{g}_{i\ov{j}} \text{~and~}
	F^{i\ov{j}} = \frac{\partial\sigma_{2}(\hat{\omega})}{\de\hat{g}_{i\ov{j}}}.
	\end{equation*}
	By \cite[(3.9)]{CHZ19}, 
	\begin{equation}\label{Gradient estimate equation 7}
	\begin{split}
	& \sum_{k}F^{i\ov{i}}\left(e_{k}(\vp_{i\ov{i}})\vp_{\ov{k}}+\ov{e}_{k}(\vp_{i\ov{i}})\vp_{k}\right) \\[1mm]
	\geq {} & -Ce^{-\vp}|\de\vp|_{g}^{2}-Ce^{-\vp}|\de\vp|_{g}+2|\de\vp|_{g}^{2}F^{i\ov{i}}\vp_{i\ov{i}}
	-2(n-1)\textrm{Re}\big(\sum_{k}(|\de\vp|_{g}^{2})_{k}\vp_{\ov{k}}\big) \\[1mm]
	& +2(n-1)|\de\vp|_{g}^{4}-\frac{n-1}{\alpha}e^{-\vp}|\de\vp|_{g}^{2}f
	+\frac{n-1}{2\alpha}e^{-\vp}\textrm{Re}\big(\sum_{k}f_{k}\vp_{\ov{k}}\big).
	\end{split}
	\end{equation}
	There are two terms containing $\frac{1}{\alpha}$, i.e.
	\begin{equation*}
	-\frac{n-1}{\alpha}e^{-\vp}|\de\vp|_{g}^{2}f
	+\frac{n-1}{2\alpha}e^{-\vp}\textrm{Re}(\sum_{k}f_{k}\vp_{\ov{k}}).
	\end{equation*}
	Recall the definition of $f$ \cite[(3.2)]{CHZ19}. It can be written as 
	\begin{equation}\label{alphaf}
	f=\alpha\tilde{f},
	\end{equation}
	where 
	\begin{align}\label{Definition of tilde f}
	\tilde{f}\omega^{n} &=  2\rho\wedge\omega^{n-1}+\alpha e^{-2\vp}\rho^{2}\wedge\omega^{n-2}-4n\mu\frac{\omega^{n}}{n!}\notag \\
	& +4n\alpha e^{-\vp}\sqrt{-1}\left(\de\vp\wedge\dbar\vp\wedge\rho-\de\vp\wedge\dbar\rho
	-\de\rho\wedge\dbar\vp+\de\dbar\rho\right)\wedge\omega^{n-2}.
	\end{align}
	Hence, 
	\begin{equation}
	\begin{split}
	&-\frac{n-1}{\alpha}e^{-\vp}|\de\vp|_{g}^{2}f
	+\frac{n-1}{2\alpha}e^{-\vp}\textrm{Re}(\sum_{k}f_{k}\vp_{\ov{k}})\\
	=&-(n-1)e^{-\vp}|\de\vp|_{g}^{2}\tilde{f}
	+\frac{n-1}{2}e^{-\vp}\textrm{Re}(\sum_{k}\tilde{f}_{k}\vp_{\ov{k}}).
	\end{split}
	\end{equation}
	
By (\ref{Definition of tilde f}),  a direct calculation shows that
	\begin{equation}\label{Gradient estimate equation 8 A}
	\begin{split}
	& -\frac{n-1}{\alpha}e^{-\vp}|\de\vp|_{g}^{2}f+\frac{n-1}{2\alpha}e^{-\vp}\textrm{Re}\big(\sum_{k}f_{k}\vp_{\ov{k}}\big) \\[1mm]
	=&-(n-1)e^{-\vp}|\de\vp|_{g}^{2}\tilde{f}
	+\frac{n-1}{2}e^{-\vp}\textrm{Re}(\sum_{k}\tilde{f}_{k}\vp_{\ov{k}})\\[1mm]
	\geq {} & -C\big(e^{-2\vp}|\de\vp|^{2}_{g}+e^{-2\vp}|\de\vp|_{g}\big)\sum_{i,j}(|e_{i}\ov{e}_{j}(\vp)|+|e_{i}e_{j}(\vp)|)-Ce^{-\vp}|\de\vp|_{g}^{4} \\
	& -Ce^{-\vp}|\de\vp|_{g}^{3}-Ce^{-\vp}|\de\vp|_{g}^{2}-Ce^{-\vp}|\de\vp|_{g} \\[1mm]
	\geq {} & -\frac{1}{10}\sum_{i,j}(|e_{i}\ov{e}_{j}(\vp)|^{2}+|e_{i}e_{j}(\vp)|^{2})-Ce^{-\vp}|\de\vp|_{g}^{4}-Ce^{-\vp}.
	\end{split}
	\end{equation}
	Here $C$ does not depend on $\frac{1}{\alpha}$.
	Substituting \eqref{Gradient estimate equation 8 A} into \eqref{Gradient estimate equation 7},
	we have
	\begin{equation}\label{Gradient estimate equation 9}
	\begin{split}
	& \sum_{k}F^{i\ov{i}}\left(e_{k}(\vp_{i\ov{i}})\vp_{\ov{k}}+\ov{e}_{k}(\vp_{i\ov{i}})\vp_{k}\right) \\[2mm]
	\geq {} & 2|\de\vp|_{g}^{2}F^{i\ov{i}}\vp_{i\ov{i}}+2(n-1)|\de\vp|_{g}^{4}
	-2(n-1)\textrm{Re}\left(\sum_{k}(|\de\vp|_{g}^{2})_{k}\vp_{\ov{k}}\right) \\[2mm]
	& -\frac{1}{10}\sum_{i,j}(|e_{i}\ov{e}_{j}(\vp)|^{2}+|e_{i}e_{j}(\vp)|^{2})-Ce^{-\vp}|\de\vp|_{g}^{4}-Ce^{-\vp}.
	\end{split}
	\end{equation}
	The rest of the proof is similar to the proof in \cite{CHZ19}.
	
\end{proof}

\begin{proposition}\label{Second order estimate}
	Let $\vp$ be a solution of (\ref{Fu-Yau equation A}) satisfying \eqref{Elliptic condition A} and $\frac{A}{M_0}\leq e^{-\vp}\leq M_0A$ for some uniform constant $M_0$. There exist uniform constants $D_{0}$ and $C_{0}$ such that if
	\begin{equation*}
	|\de\dbar\vp|_{g} \leq D, ~~D_{0}\leq D \text{~and~} A\leq A_{D}:=\frac{1}{C_{0}M_{0}D},
	\end{equation*}
	then
	\begin{equation*}
	|\de\dbar\vp|_{g} \leq \frac{D}{2}.
	\end{equation*}
\end{proposition}
\begin{proof}
	The  proof is similar to \cite[Proposition 4.1]{CHZ19}. It suffices to show that the \textquotedblleft$C$\textquotedblright, from \cite[(4.6)]{CHZ19} to \cite[(4.7)]{CHZ19}, do not depend on $\frac{1}{\alpha}$.
	
	First, we recall the notation in \cite{CHZ19}.
	Consider the following quantity
	\begin{equation*}
	Q = |\de\dbar\vp|_{g}^{2} + B|\de\vp|_{g}^{2},
	\end{equation*}
	where $B>1$ is a uniform constant to be determined later. Assume that $Q(x_{0})=\max_{M}Q$. Choose a local $g$-unitary frame $\{e_{i}\}_{i=1}^{n}$ for $T_{\mathbb{C}}^{(1,0)}M$  in a neighborhood $x_{0}$  such that $\ti{g}_{i\ov{j}}(x_{0})$ is diagonal.  Denote
	\begin{equation*}
	\hat{\omega} = e^{-\vp}\ti{\omega},
	\hat{g}_{i\ov{j}} = e^{-\vp}\ti{g}_{i\ov{j}},
	F^{i\ov{j}} = \frac{\de{\sigma_{2}(\hat{\omega})}}{\de\hat{g}_{i\ov{j}}}
	\text{~and~}
	F^{i\ov{j},k\ov{l}} = \frac{\de^{2}{\sigma_{2}(\hat{\omega})}}{\de\hat{g}_{i\ov{j}}\de\hat{g}_{k\ov{l}}}.
	\end{equation*}
	It follows
	\begin{equation*}
	F^{i\ov{j}} = \delta_{ij}F^{i\ov{i}} = \delta_{ij}e^{-\vp}\sum_{k\neq i}\ti{g}_{k\ov{k}}
	\end{equation*}
	and
	\begin{equation*}
	F^{i\ov{j},k\ov{l}}=
	\left\{\begin{array}{ll}
	1, & \text{if $i=j$, $k=l$, $i\neq k$;}\\[1mm]
	-1, & \text{if $i=l$, $k=j$, $i\neq k$;}\\[1mm]
	0, & \text{\quad\quad~otherwise.}
	\end{array}\right.
	\end{equation*}
	By the assumption of Proposition \ref{Second order estimate},  assume that
	\begin{equation}\label{Second order estimate equation 6}
	e^{-\vp}|\de\dbar\vp|_{g}\leq \frac{1}{1000n^{3}(|\alpha|+1)B}.
	\end{equation}
	
Then
	\begin{equation}\label{Second order estimate equation 1}
	|F^{i\ov{i}}-(n-1)| \leq \frac{1}{100}
	\text{~and~} |F^{i\ov{j},k\ov{l}}|\leq 1.
	\end{equation}
	By \cite[(4.4)]{CHZ19},
	\begin{equation}\label{Second order estimate equation 2}
	F^{i\ov{j}}e_{k}\ov{e}_{l}(e^{-\vp}\ti{g}_{i\ov{j}})
	= I_{1}+I_{2}+I_{3},
	\end{equation}
	where 
	\begin{equation*}
	\begin{split}
	I_{1} = {} & -F^{i\ov{j},p\ov{q}}e_{k}(e^{-\vp}\ti{g}_{i\ov{j}})\ov{e}_{l}(e^{-\vp}\ti{g}_{p\ov{q}}), \\[1.5mm]
	I_{2} = {} & -2n(n-1)\alpha e_{k}\ov{e}_{l}(e^{-\vp}|\de\vp|_{g}^{2}), \\
	I_{3} = {} & \frac{n(n-1)}{2}e_{k}\ov{e}_{l}(e^{-2\vp}f).
	\end{split}
	\end{equation*}	
In the following,	we deal with each term in (\ref{Second order estimate equation 2}) below.  
Recall	\begin{equation}\label{ddbar formula}
\vp_{i\ol{j}}=\ddbar\vp(e_{i}, \ol{e}_{j})=e_{i}\ol{e}_{j}(\vp)-[e_{i}, \ol{e}_{j}]^{(0,1)}(\vp).
\end{equation}
 For $I_{1}$, by (\ref{Second order estimate equation 1}), Proposition \ref{Gradient estimate} and the Cauchy-Schwarz inequality, we derive
	\begin{equation*}
	\begin{split}
	|I_{1}| \leq {} & \sum_{i,j,k}\left|e_{k}(\alpha e^{-2\vp}\rho_{i\ov{j}}+2n\alpha e^{-\vp}\vp_{i\ov{j}})\right|^{2} \\
	\leq {} & 2\sum_{i,j,k}\left|e_{k}(2n\alpha e^{-\vp}\vp_{i\ov{j}})\right|^{2}
	+2\sum_{i,j,k}\left|e_{k}(\alpha e^{-2\vp}\rho_{i\ov{j}})\right|^{2} \\
	\leq {} & 8n^{2}\alpha^{2}e^{-2\vp}\sum_{i,j,k}\left|e_{k}e_{i}\ov{e}_{j}(\vp)-e_{k}[e_{i},\ov{e}_{j}]^{(0,1)}(\vp)
	-\vp_{k}\vp_{i\ov{j}}\right|^{2}+C\alpha^{2}e^{-4\vp} \\
	\leq {} & \alpha^{2}\left(16n^{2}e^{-2\vp}\sum_{i,j,k}|e_{k}e_{i}\ov{e}_{j}(\vp)|^{2}
	+Ce^{-2\vp}\sum_{i,j}\left(|e_{i}\ov{e}_{j}(\vp)|^{2}+|e_{i}{e}_{j}(\vp)|^{2}\right)+Ce^{-2\vp}\right).
	\end{split}
	\end{equation*}
 Similarly,  for $I_{2}$ and $I_{3}$, by \eqref{alphaf}, we get
	\begin{equation*}
	|I_{2}| \leq \alpha\left(Ce^{-\vp}\sum_{i,j,p}|e_{p}e_{i}\ov{e}_{j}(\vp)|
	+Ce^{-\vp}\sum_{i,j}\left(|e_{i}\ov{e}_{j}(\vp)|^{2}+|e_{i}{e}_{j}(\vp)|^{2}\right)+Ce^{-\vp}\right)
	\end{equation*}
	and
	\begin{equation*}
	\begin{split}
	|I_{3}|    = {} & \frac{n(n-1)}{2}e^{-2\vp}\left|4\vp_{k}\vp_{\ov{l}}f-2e_{k}\ov{e}_{l}(\vp)f
	-2\vp_{\ov{l}}f_{k}-2\vp_{k}f_{\ov{l}}+e_{k}\ov{e}_{l}f\right| \\
	={}& \alpha\left(\frac{n(n-1)}{2}e^{-2\vp}\left|4\vp_{k}\vp_{\ov{l}}\tilde{f}-2e_{k}\ov{e}_{l}(\vp)\tilde{f}
	-2\vp_{\ov{l}}\tilde{f}_{k}-2\vp_{k}\tilde{f}_{\ov{l}}+e_{k}\ov{e}_{l}\tilde{f}\right|\right)\\
	\leq {} & \alpha\left(Ce^{-2\vp}\sum_{i,j,p}|e_{p}e_{i}\ov{e}_{j}(\vp)|+Ce^{-2\vp}\sum_{i,j}\left(|e_{i}\ov{e}_{j}(\vp)|^{2}
	+|e_{i}{e}_{j}(\vp)|^{2}\right)+Ce^{-2\vp}\right),
	\end{split}
	\end{equation*}
	where we   used Proposition \ref{Gradient estimate} and (\ref{Definition of tilde f}).  Thus substituting these estimates into (\ref{Second order estimate equation 2}),   we conclude
	\begin{equation}\label{Second order estimate equation 7}
	\begin{split}
	&|F^{i\ov{i}}e_{k}\ov{e}_{l}(e^{-\vp}\ti{g}_{i\ov{i}})|\\
	&\leq  16n^{2}\alpha^{2}e^{-2\vp}\sum_{i,j,p}|e_{p}e_{i}\ov{e}_{j}(\vp)|^{2}+C\alpha e^{-\vp}\sum_{i,j,p}|e_{p}e_{i}\ov{e}_{j}(\vp)| \\
	& +C\alpha e^{-\vp}\sum_{i,j}(|e_{i}\ov{e}_{j}(\vp)|^{2}+|e_{i}e_{j}(\vp)|^{2})+C\alpha e^{-\vp}.
	\end{split}
	\end{equation}
	
In addition,  by the definition of $\ti{g}_{i\ov{i}}$ and (\ref{ddbar formula}),  we have
	\begin{equation*}
	\begin{split}
	F^{i\ov{i}}e_{k}\ov{e}_{l}(e^{-\vp}\ti{g}_{i\ov{i}})
	= {} & \alpha F^{i\ov{i}}e_{k}\ov{e}_{l}(e^{-2\vp}\rho_{i\ov{i}})+2n\alpha F^{i\ov{i}}e_{k}\ov{e}_{l}(e^{-\vp}\vp_{i\ov{i}}) \\
	= {} & \alpha F^{i\ov{i}}e_{k}\ov{e}_{l}(e^{-2\vp}\rho_{i\ov{i}})+2n\alpha F^{i\ov{i}}e_{k}\ov{e}_{l}(e^{-\vp}e_{i}\ov{e}_{i}(\vp)) \\
	& -2n\alpha F^{i\ov{i}}e_{k}\ov{e}_{l}(e^{-\vp}[e_{i},\ov{e}_{i}]^{(0,1)}(\vp)).
	\end{split}
	\end{equation*}
	Then by  (\ref{Second order estimate equation 1}) and Proposition \ref{Gradient estimate},  it  follows that
	\begin{equation}
	\begin{split}
	|2n\alpha e^{-\vp}F^{i\ov{i}}e_{k}\ov{e}_{l}e_{i}\ov{e}_{i}(\vp)|
	\leq {} & |F^{i\ov{i}}e_{k}\ov{e}_{l}(e^{-\vp}\ti{g}_{i\ov{i}})|+C\alpha e^{-\vp}\sum_{i,j,p}|e_{p}e_{i}\ov{e}_{j}(\vp)| \notag\\
	& +C\alpha e^{-\vp}\sum_{i,j}(|e_{i}\ov{e}_{j}(\vp)|^{2}+|e_{i}e_{j}(\vp)|^{2})+C\alpha e^{-\vp}.
	\end{split}
	\end{equation}
	Thus substituting (\ref{Second order estimate equation 7}) into the above inequality, we  derive
	\begin{equation}\label{Second order estimate equation 5}
	\begin{split}
	|F^{i\ov{i}}e_{k}\ov{e}_{l}e_{i}\ov{e}_{i}(\vp)|
	\leq {} & 8n|\alpha|e^{-\vp}\sum_{i,j,p}|e_{p}e_{i}\ov{e}_{j}(\vp)|^{2}+C\sum_{i,j,p}|e_{p}e_{i}\ov{e}_{j}(\vp)| \\
	& +C\sum_{i,j}(|e_{i}\ov{e}_{j}(\vp)|^{2}+|e_{i}e_{j}(\vp)|^{2})+C.
	\end{split}
	\end{equation}
	Here $C$ does not depend on $\frac{1}{\alpha}.$ The rest of the proof is similar to the proof of \cite[Proposition 4.1]{CHZ19}.
\end{proof}

Then we have
\begin{proposition}\label{first and second order}
	Let $\vp_{\epsilon}$ be the solution of \eqref{Fu-Yau equation 1}.
	Assume that
	\begin{equation*}
	\frac{A\epsilon^4}{M_0}\leq e^{-\vp_{\epsilon}}\leq M_0A\epsilon^{4}  \text{~and~} |\de\dbar\vp_{\epsilon}|_{g}\leq D.
	\end{equation*}
	Then there exists a uniform constant $C_{0}$ such that if
	\begin{equation*}
	A \leq A_{D} := \frac{1}{C_{0}M_{0}D},
	\end{equation*}
	then
	\begin{equation*}
	|\de\vp_{\epsilon}|_{\omega_{\epsilon}}^{2} \leq M_{1}, |\ddbar\vp_{\epsilon}|_{\omega_{\epsilon}}\leq \frac{D}{2},
	\end{equation*}
	where $M_{1}$ is a uniform constant only depending on $K$ and $(M,\omega)$.	
\end{proposition}
\begin{proof}

Replacing $\omega$, $\vp$,  $\rho$, $\alpha$ and	$\mu$ by $\omega_{\epsilon}, $ $\tilde{\vp}_{\epsilon}$, $\tilde{\rho}_{\epsilon}$, $\tilde{\alpha}$ and $\tilde{\mu}_{\epsilon}$ in the proof of  Proposition \ref{Gradient estimate} and  Proposition \ref{Second order estimate}, we obtain it.
\end{proof}

\subsection{Proof of Theorem \ref{Estimate}}Now, we complete the proof of Theorem \ref{Estimate}.

\begin{proof}
We consider the family equation:

\begin{equation}\label{Fu-Yau equation c}
\begin{split}
\ddbar(e^{\tilde{\vp}_{\epsilon}}\omega_{\epsilon}-s\tilde{\alpha} e^{-\tilde{\vp}_{\epsilon}}\tilde{\rho}_{\epsilon}\omega_{\epsilon}) \wedge\omega_{\epsilon}^{n-2}
&+sn\tilde{\alpha}(\ddbar\tilde{\vp}_{\epsilon})^{2}\wedge\omega_{\epsilon}^{n-2}\\
&+s\tilde{\mu}_{\epsilon}\omega_{\epsilon}^{n}=0.
\end{split}
\end{equation}

 We define the following sets of functions on $M$,
\begin{equation*}
\begin{split}
B & = \{ u\in C^{2,\beta}(M) ~|~ \int_{M}e^{-u}\omega_{\epsilon}^{n}=A\epsilon^{2n_{1}} \},\\[2mm]
B_{1} & = \{ (u,s)\in B\times[0,1] ~|~ \text{$u$ ~satisfies~}  e^{u}\omega_{\epsilon}+\tilde{\alpha} e^{-u}\ti{\rho}_{\epsilon}+2n\tilde{\alpha}\ddbar u \in \Gamma_{2}(M) \}.\\
\end{split}
\end{equation*}
Note that at $s=0$, $\tilde{\vp}_{\epsilon}=\ln A $.
Let $I$ be the set
\begin{equation*}
\{ s\in [0,1] ~|~ \text{there exists $(\tilde{\vp},s)\in B_{1}$ satisfying \eqref{Fu-Yau equation c}} \}.
\end{equation*}

By the \cite[Subsection 5.1]{CHZ19}, $I$ is open. 
Now we prove $I$ is closed. The proof is similar to \cite[Proposition 5.1]{CHZ19}. For reader's convenience, we include the proof.	Let $\tilde{\vp}_{s\epsilon}$ be the solution of (\ref{Fu-Yau equation c}) at time $s$. Without confusion, we use $\tilde{\vp}_{s}$ to denote $\tilde{\vp}_{s\epsilon}$. We have
  
\begin{proposition}\label{A priori estimate}
	 If $\vp_{s}$ satisfies (\ref{Elliptic condition}) and (\ref{Normalization condition}), there exists a constant $C_{A}$ depending only on $A$, $K$, $\beta$ and $(M,\omega_{\epsilon})$ such that
	\begin{equation*}
	\begin{split}
	&\sup_{M}e^{-\tilde{\vp}_{s}} \leq 2M_{0}A, \,   \sup_{M}|\de\dbar\tilde{\vp}_{s}|_{\epsilon} \leq D_{0}, \\
	&	\sup_{M}|\de\tilde{\vp}_{s}|_{\epsilon}^{2} \leq C,\ \ 	\|\tilde{\vp}_{s}\|_{C^{2,\beta}}\leq C_{A}, ~s\in[0,s_{0}).
	\end{split}
	\end{equation*}
Here $C, M_{0}, D_{0}$ are uniform constant only depending on $(M,\omega)$ and $K$.	
\end{proposition}

\begin{proof} First, we  give the zero order estimate.   In fact,  we have
	\begin{claim}\label{claim-2}
		\begin{equation}\label{Closeness equation 1}
		\sup_{M}e^{-\tilde{\vp}_{s}} \leq 2M_{0}A, ~s\in[0,s_{0}),
		\end{equation}
		where $M_{0}$ is the constant in Proposition \ref{C0 estimate}.
	\end{claim}
	Note that  $\tilde{\vp}_{0}=-\ln A$.  Then $\sup_{M}e^{-\tilde{\vp}_{0}} \leq M_{0}A$, which satisfies (\ref{Closeness equation 1}).   If (\ref{Closeness equation 1}) is false, we can find  $\tilde{s}\in (0,s_{0})$ such that
	\begin{equation}\label{Closeness equation 2}
	\sup_{M}e^{-\tilde{\vp}_{\tilde{s}}} = 2M_{0}A.
	\end{equation}
 Recall $\delta_{0}=\sqrt{\frac{1}{2|\tilde{\alpha}|\|\tilde{\rho}\|_{C^{0}(\omega_{\epsilon})}+1}}$ is chosen  as in Proposition \ref{C0 estimate}. Assume $2M_{0}A\le   \delta_{0}$,  then $e^{-\vp_{\ti{s}}}\leq\delta_{0}$. Applying  Proposition \ref{C0 estimate} to $\tilde{\vp}_{\ti{s}}$,  we derive
$e^{-\tilde{\vp}_{\tilde{s}}} \leq M_{0}A,$
	which contradicts to (\ref{Closeness equation 2}). Then we obtain (\ref{Closeness equation 1}). Combining (\ref{Closeness equation 1}) and Proposition \ref{C0 estimate},  the zero order estimate follows.

Using the similar argument, we obtain the second order estimate
	\begin{equation}\label{Closeness equation 4}
	\sup_{M}|\de\dbar\tilde{\vp}_{s}|_{g} \leq D_{0},
	\end{equation}
	for any $s\in(0,s_{0})$, where $D_{0}$ is the constant as in Proposition \ref{Second order estimate}.
	
	By (\ref{Closeness equation 4}) and Proposition \ref{first and second order}, we have  the first order estimate
	\begin{equation}\label{Closeness equation 3}
	\sup_{M}|\de\tilde{\vp}_{s}|_{g}^{2} \leq C.
	\end{equation}
	On the other hand,  \eqref{Fu-Yau equation 2} can be written as a $2$-nd Hessian equation of the form (see \cite[1.8]{CHZ19})
	\begin{align}\label{fu-2}
	\sigma_{2}(\ti{\omega})=F(z,\tilde{\varphi},  \de\tilde{\vp}),
	\end{align}
	where

	$$F(z,\tilde{\varphi},  \de\tilde{\vp})= \frac{n(n-1)}{2}\left(e^{2\tilde{\vp}}-4\alpha e^{\tilde{\varphi}}|\de\tilde{\vp}|_{\epsilon}^{2}\right)+\frac{n(n-1)}{2}f(z,\tilde{\vp},\de\tilde{\vp}).$$
In fact,	 $f(z,\tilde{\vp},\de\tilde{\vp})$  satisfies
	$$|f(z,\tilde{\vp},\de\tilde{\vp})|\le C(e^{-2\tilde{\varphi}}+ e^{-\tilde{\varphi}}|\partial\tilde{\varphi}|^2_{\epsilon}+1).$$
	
By (\ref{Closeness equation 4}), (\ref{Closeness equation 3}) and (\ref{Fu-Yau equation c}),  we get
	\begin{equation}\label{non-degenrate}
	\left|\sigma_{2}(\ti{\omega})-\frac{n(n-1)}{2}e^{2\tilde{\vp}}\right| \leq Ce^{\tilde{\vp}}.
	\end{equation}
Using the zero order estimate, we deduce $\frac{1}{CA^{2}} \leq \sigma_{2}(\ti{\omega}) \leq \frac{C}{A^{2}}.$
	Hence,   (\ref{Fu-Yau equation c})  is uniformly elliptic and non-degenerate.  
	
	By  the $C^{2,\alpha}$-estimate (cf. \cite[Theorem 1.1]{TWWY15}), we obtain
$	\|\vp_{s}\|_{C^{2,\beta}}\leq C_{A}.$
	\end{proof}

Then $I=[0,1]$. By Proposition \ref{A priori estimate}, the Theorem \ref{Estimate} follows.

\end{proof}

\section{Proof of Theorem \ref{convergence}}
In this section, we prove the Theorem \ref{convergence}. First, we have the following proposition. It plays an important role in the proof of Theorem \ref{convergence}.
\begin{proposition}\label{near the average}
	Let $\underline{\tilde{\vp}}_{\epsilon}=\frac{1}{\vol(M_{1})}\int_{M_{1}}\tilde{\vp}_{\epsilon}\omega_{1}^{n_{1}}$.
	Then
	\begin{equation}\label{average}
	|\tilde{\vp}_{\epsilon}-\underline{\tilde{\vp}}_{\epsilon}|\leq C\epsilon^{2},
	\end{equation}
	where $C$ is a uniform constant only depending on $(M,\omega)$ and $K$.
\end{proposition}
\begin{proof}
	Let $\tilde{\psi}=\frac{1}{\epsilon^{2}}(\tilde{\vp}_{\epsilon}-\underline{\tilde{\vp}}_{\epsilon})$. By Theorem \ref{Estimate}, we can choose $A$ small enough such that  $e^{\tilde{\vp}}\omega_{\epsilon}+\ddbar\tilde{\vp}_{\epsilon}>0.$
	Then, by Theorem \ref{Estimate}, we have
	\begin{equation*}
	(e^{\tilde{\vp}}\omega_{1}+\ddbar\tilde{\psi})^{n}\leq C.
	\end{equation*}
	Using the $L^{\infty}$ estimate for complex Monge-Amp\`ere equation \cite{KN15}, we have $\|\tilde{\psi}\|_{L^{\infty}}\leq C,$ which implies
	$|\tilde{\vp}_{\epsilon}-\underline{\tilde{\vp}}_{\epsilon}|\leq C\epsilon^{2}.$
\end{proof}
Now we prove the Theorem \ref{convergence}. 
\begin{proof}[Proof of Theorem \ref{convergence}]
		By Theorem \ref{Estimate}, there exists a subsequence, still denoted by $\ddbar\tilde{\vp_{\epsilon}}$, and a function $\hat{\vp}$  such that $\tilde{\vp}_{\epsilon}$ converge to $\hat{\vp}$ in $C^{1,\beta}(M,\omega)$ and $\ddbar\tilde{\vp}_{\epsilon}$ weakly converge to $\ddbar\hat{\vp}$. By Proposition \ref{near the average},  $\hat{\vp}$ is a function on $M_{2}.$
		Note that $\tilde{\mu}_{\epsilon}$ are uniformly bounded. There exists a subsequence, still denoted by $\tilde{\mu}_{\epsilon}$, weakly converges to a function $\mu_{0}$ in  $L^{2}(M,\omega)$.
	
	Let $\eta$ be a smooth function on $M_{2}$. Then, by \eqref{Fu-Yau equation 2}, we have
	\begin{equation}\label{act on test function}
	\begin{split}
	\frac{1}{\epsilon^{2n_{1}}}&\int_{M}\eta\ddbar{e^{\tilde{\vp}_{\epsilon}}}\wedge\omega_{\epsilon}^{n-1}-	\frac{1}{\epsilon^{2n_{1}-2t}}\alpha\int_{M}\eta \ddbar(e^{-2\tilde{\vp}_{\epsilon}}\rho_{\epsilon})\wedge\omega^{n-2}_{\epsilon}\\
	&+\frac{1}{\epsilon^{2n_{1}-t}}\int_{M}n\alpha\eta(\ddbar{\tilde{\vp_{\epsilon}}})^{2}\wedge\omega^{n-2}_{\epsilon}+\frac{1}{\epsilon^{2n_{1}-t}}\int_{M}\mu_{\epsilon}\eta\omega^{2n}_{\epsilon}=0.
	\end{split}
	\end{equation}
First, we estimate the first term. Note
\begin{equation}
\begin{split}
 &\frac{1}{\epsilon^{2n_{1}}}\int_{M}\eta\ddbar{e^{\tilde{\vp}_{\epsilon}}}\wedge\omega_{\epsilon}^{n-1}\\
=&\frac{1}{\epsilon^{2n_{1}}}\int_{M}\eta\ddbar(e^{\tilde{\vp}_{\epsilon}}-e^{\underline{\tilde{\vp}}_{\epsilon}})\wedge\omega_{\epsilon}^{n-1}+\frac{1}{\epsilon^{2n_{1}}}\int_{M}\eta\ddbar{e^{\underline{\tilde{\vp}}_{\epsilon}}}\wedge\omega_{\epsilon}^{n-1}\\
=&\frac{1}{\epsilon^{2n_{1}}}\int_{M}(e^{\tilde{\vp}_{\epsilon}}-e^{\underline{\tilde{\vp}}_{\epsilon}})\ddbar\eta\wedge\omega_{\epsilon}^{n-1}+\frac{1}{\epsilon^{2n_{1}}}\int_{M}e^{\underline{\tilde{\vp}}_{\epsilon}}\ddbar\eta\wedge\omega_{\epsilon}^{n-1}\\
=&\int_{M}(e^{\tilde{\vp}_{\epsilon}}-e^{\underline{\tilde{\vp}}_{\epsilon}})\ddbar\eta\wedge\omega_{1}^{n_{1}}\wedge\omega_{2}^{n_{2}-1}
  +\int_{M}e^{\underline{\tilde{\vp}}_{\epsilon}}\ddbar\eta\wedge\omega_{1}^{n_{1}}\wedge\omega_{2}^{n_{2}-1}\\
:=&A_{1}+A_{2}.
\end{split}
\end{equation}

For the first term on the right hand side, using Theorem \ref{Estimate} and \eqref{average}, we have
\[A_{1}\rightarrow0.\]
Note $\tilde{\vp}_{\epsilon}$  converge to $\hat{\vp}$. 
Combining with \eqref{average}, we have
\[\underline{\tilde{\vp}}_{\epsilon}\rightharpoonup\hat{\vp}.\]
Hence, we have
\[A_{2}\rightarrow\int_{M}e^{\hat{\vp}}\ddbar \eta \wedge \omega_{1}^{n_{1}}\wedge\omega_{2}^{n_{2}-1}.\]
Therefore,
\begin{equation}\label{test term 1}
 \frac{1}{\epsilon^{2n_{1}}}\int_{M}\eta\ddbar{e^{\tilde{\vp}_{\epsilon}}}\wedge\omega_{\epsilon}^{n-1}\rightarrow \int_{M}e^{\hat{\vp}}\ddbar \eta \wedge \omega_{1}^{n_{1}}\wedge\omega_{2}^{n_{2}-1}.
\end{equation}
Denote by $\Omega^{1,,1}(M_{1}), \Omega^{1,,1}(M_{2})$  the set of smooth sections of  $(1,1)$ form on $M_{1}$, $M_{2}$, respectively.  Let $\rho_{1\epsilon}$, $\rho_{2\epsilon}$ be the projection  of $\rho_{\epsilon}$ into $\Omega^{1,1}(M_{1})$, $\Omega^{1,1}(M_{2})$, respectively.
By \eqref{uniform function}, we have
\begin{equation}\label{second term convergence}
\begin{split}
 &\frac{1}{\epsilon^{2n_{1}-2t}}\alpha\int_{M}\eta \ddbar(e^{-\tilde{\vp}_{\epsilon}}\rho_{\epsilon})\wedge \omega^{n-2}_{\epsilon}\\
=&\frac{1}{\epsilon^{2n_{1}-2t}}\alpha\int_{M}e^{-\tilde{\vp}_{\epsilon}}\rho_{\epsilon}\wedge \ddbar\eta\wedge\omega^{n-2}_{\epsilon}\\
=&\epsilon^{2t-2}\alpha\int_{M}e^{-\tilde{\vp}_{\epsilon}}\rho_{1\epsilon}\wedge \ddbar\eta\wedge \omega_{1}^{n_{1}-1}\wedge\omega_{2}^{n_{2}-1}\\
&+\epsilon^{2t}\alpha\int_{M}e^{-\tilde{\vp}_{\epsilon}}\rho_{2\epsilon}\wedge \ddbar\eta\wedge \omega_{1}^{n_{1}}\wedge\omega_{2}^{n_{2}-2}\rightarrow 0.\\
\end{split}
\end{equation}
 
For the third term in \eqref{act on test function},
by Theorem \ref{Estimate} , we have
\[|\frac{1}{\omega^{n}\epsilon^{2n_{1}-2t}}(\ddbar{\tilde{\vp}_{\epsilon}})^{2}\wedge\omega^{n-2}_{\epsilon}|\leq C\epsilon^{2t},\]
which implies
\begin{equation}\label{test term 2}
\begin{split}
\frac{1}{\epsilon^{2n_{1}-t}}\int_{M}n\alpha\eta(\ddbar{\tilde{\vp}_{\epsilon}})^{2}\wedge\omega^{n-2}_{\epsilon}\rightarrow 0.
\end{split}
\end{equation}

For the last term,  we have
\begin{equation*}
\frac{1}{\epsilon^{2n_{1}}}\int_{M}\tilde{\mu}_{\epsilon}\eta\omega^{2n}_{\epsilon}\rightarrow \int_{M_{2}}\eta (\int_{M_{1}}\mu_{0}\omega_{1}^{n_{1}})\omega_{2}^{n_{2}}.
\end{equation*}

In conclusion, using \eqref{act on test function}, \eqref{test term 1}, \eqref{second term convergence} and \eqref{test term 2}, when $\epsilon\rightarrow 0$, we obtain
 \begin{equation}
 \begin{split}
 \int_{M}e^{\hat{\vp}}\ddbar \eta\wedge\omega_{1}^{n_{1}}\wedge\omega_{2}^{n_{2}-1}+\int_{M_{2}}\eta (\int_{M_{1}}\mu_{0}\omega_{1}^{n_{1}})\omega_{2}^{n_{2}}\\
=\int_{M_{1}}\omega_{1}^{n}\int_{M_{2}}e^{\hat{\vp}}\ddbar \eta\wedge\omega_{2}^{n_{2}-1}++\int_{M_{2}}\eta (\int_{M_{1}}\mu_{0}\omega_{1}^{n_{1}})\omega_{2}^{n_{2}}=0.
 \end{split}
\end{equation}
Recall $\hat{\vp}$ is independent of the points on $M_{1}$.
It implies
\[\int_{M_{2}}e^{\hat{\vp}}\ddbar \eta\wedge\omega_{2}^{n_{2}-1}+\int_{M_{2}}\eta \hat{\mu}\omega_{2}^{n_{2}} =0,\]
where $\hat{\mu}=\frac{1}{\int_{M_{1}}\omega_{1}^{n_{1}}}\int_{M_{1}}\mu_{0}\omega^{n_{1}}_{1}.$
Therefore, $\hat{\vp}$ is a weak solution of
\begin{equation}
\ddbar e^{\hat{\vp}}\wedge \omega_{2}^{n_{2}-1}+\hat{\mu}\omega_{2}^{n_{2}}=0.
\end{equation}
\end{proof}

\end{document}